\documentclass[11pt,a4paper]{amsart} 
\usepackage{geometry}
\usepackage{tabls}
\usepackage{url}
\usepackage[utf8]{inputenc}
\usepackage{amsfonts}
\usepackage{amssymb}
\usepackage{mathrsfs}
\usepackage{amsmath}
\usepackage{hyperref}
\usepackage{amscd}
\usepackage{fancyhdr}
\usepackage{enumerate}
\usepackage{paralist}
\usepackage{xypic}
\usepackage{graphicx}
\usepackage{cite}
\usepackage{lipsum}
\usepackage{footmisc}
\allowdisplaybreaks

\numberwithin{equation}{section}
\theoremstyle{plain}
\newtheorem{theorem}{Theorem}[section]
\newtheorem{lemma}[theorem]{Lemma}
\newtheorem{corollary}[theorem]{Corollary}
\newtheorem{remark}[theorem]{Remark}
\numberwithin{equation}{section}

\def\li{\langle}
\def\ri{\rangle}

\def\R{\mathbb{R} }

\def\fg{\mathfrak{g}}
\def\gl{\mathfrak{gl}}

\def\df{(dF)_{(\tau,A,B,C)}}

\begin{document}
\title{A simple perturbation of Vafa-Witten equations and a transversality result}
\author{Bo Dai}
\address{School of Mathematical Sciences,
	Peking University, Beijing 100871, China}
\email{daibo@math.pku.edu.cn}

\author{Ren Guan}
\address{School of Mathematics and Statistics, Jiangsu Normal University, Xuzhou 221100, China}
\email{guanren@jsnu.edu.cn}

\begin{abstract}
We consider a simple perturbation of the Vafa-Witten equations, and prove that for generic perturbation parameter, the full rank part of the perturbed Vafa-Witten moduli space satisfies transversality condition, when the structure group is $SU(2)$ or $SO(3)$. 
\end{abstract}

\maketitle

\section{Introduction}

The Vafa-Witten equations came from a topological twist of $\mathcal{N}=4$ supersymmetric Yang-Mills theory. Vafa and Witten conjectured that the partition function of their theory transforms as modular form under $S$-duality, an important duality in quantum field theory. 

Let $X$ be a smooth, oriented, Riemannian 4-manifold, $P$ be a principal bundle over $X$ with structure group $G$, $\fg_P$ be the adjoint bundle of $P$. The Vafa-Witten equations have the following form:
\begin{equation}\left\{
	\begin{aligned}
		&d_A^*B +d_A C=0,\\&F_A^++\frac{1}{8}[B\centerdot B]+\frac{1}{2}[B,C]=0,  \\
	\end{aligned}\right. \label{aa}
\end{equation}
where $A$ is a connection of the principal bundle $P$, $B$ is a self-dual 2-form with values in the adjoint bundle $\fg_P$, $C$ is a section of the adjoint bundle $\fg_P$, and $[B\centerdot B]$ is a combination of the Lie bracket of $\fg$ and the Lie algebra structure of $\Lambda^{2,+}\R^4 \cong so(3)$.  The Vafa-Witten equations are invariant under the action of  gauge group. The Vafa-Witten moduli space is the space of solutions to the Vafa-Witten equations modulo gauge transformation. 
 
 Mares and Taubes did fundamental work to the analytic aspect of Vafa-Witten theory,\cite{Ma,Tau}, but key challenges remain, such as the compactification and transversality of the Vafa-Witten moduli spaces. Tanaka and Thomas made key contribution to the algebraic aspect of Vafa-Witten theory. They defined Vafa-Witten invariants for projective surfaces by the method of derived algebraic geometry and virtual intersection theory, and verified $S$-duality in some cases, \cite{TT1,TT2}. Since then, the algebraic aspect of Vafa-Witten theory developed rapidly. 
 
 In this short note, we consider a simple perturbed version of the Vafa-Witten equations, and obtain a transversality result of the perturbed Vafa-Witten moduli spaces.

The perturbed Vafa-Witten equations on $X$ are
\begin{equation}\left\{
		\begin{aligned}
			&d_A^*B +d_A C=0,\\&F_A^++\frac{1}{8}[B\centerdot B]+\frac{1}{2}[B,C]+\tau B=0,  \\
		\end{aligned}\right. \label{ab}
	\end{equation}
where the triple $(A,B,C)$ is as above,  $\tau$ is a section of the bundle $\gl(\Lambda^{2,+}T^*X)$. We note that
$\tau$ acts on $\fg_P$ trivially, i.e., as identity.
Thus the perturbed Vafa-Witten equations are invariant under gauge group action,
and the perturbed moduli space $\mathcal{M}_{VW,\tau}(P)$ is well-defined.

To apply analytic tools, we need to work on appropriate Sobolev spaces. Let $$\begin{aligned}
	\mathcal{C}_k(P)&:=\mathcal{A}_k (P)\times L_k^2(X,\fg_P\otimes\Lambda^{2,+})\times L_k^2(X, \fg_P),\\
	\mathcal{C}_{k-1}'(P)&:=L_{k-1}^2(X,\fg_P\otimes\Lambda^1)\times L_{k-1}^2(X,\fg_P\otimes\Lambda^{2,+}),
\end{aligned}$$
be the configuration spaces, $C^r(X,\gl(\Lambda^{2,+}))$ be the parameter space, where $r>k\geq2$ are integers. Let $\mathcal{G}_{k+1}(P)$ be the group of $L_{k+1}^2$ gauge transformations. In the remaining part of this note, we take the structure group $G$ to be $SU(2)$ or $SO(3)$, and follow notations in \cite{DG}.

\begin{theorem}\label{m}
    Let $X$ be a closed, oriented, smooth, Riemannian 4-manifold, and $P\to X$ be a principal $SU(2)$- or $SO(3)$-bundle. Then for generic perturbation $\tau\in C^r(X,\gl(\Lambda^{2,+}))$, where $r>k\geq 2$ are integers, the full rank part of the moduli space, $\mathcal{M}_{VW,\tau}^{(3)}(P)$, is a smooth, oriented, zero-dimensional manifold in the quotient space $\mathcal{B}_k(P)=\mathcal{C}_k(P)/\mathcal{G}_{k+1}(P)$.
\end{theorem}

\begin{remark}
	Other kind of perturbations of Vafa-Witten equations had been studied in \cite{G1,G2,Tan1}.
\end{remark}

\begin{remark}
A PhD student, Mr. Zhengxiong Cao from Humboldt-Berlin,
pointed out an error in our earlier paper \cite{DG}. That is, the three equations in \cite{DG} (4.3) do not hold separately. 
Theorem~\ref{m} is a replacement of Theorems 1, 4 and 5 in \cite{DG}. 
\end{remark}

\section{Prove of the main theorem}

Consider the parameterized Vafa-Witten map
\begin{equation}\begin{aligned}
F:C^r(X,\gl(\Lambda^{2,+}))\times\mathcal{C}_k(P)&\to\mathcal{C}_{k-1}'(P),\\
			F(\tau,A,B,C)&:=\begin{pmatrix}
				d_A^*B+d_AC\\
				F_A^++\frac{1}{8}[B\centerdot B]+\frac{1}{2}[B,C]+\tau B\\
			\end{pmatrix}.
		\end{aligned} \label{ac}
	\end{equation}
\begin{lemma} \label{sj}
    With notation in \eqref{ac}, at any point $(\tau,A,B,C)$ where $F(\tau,A,B,C)=0$ and $B$ has rank 3, the differential $dF$ at this point is surjective.
\end{lemma}
\begin{proof}
    The differential $(dF)_{(\tau,A,B,C)}$ differs from $d^1_{(A,B,C)}$ in \cite{Ma} page 28 slightly, given by
$$
(dF)_{(\tau,A,B,C)}:C^r(X,\gl(\Lambda^{2,+}))\oplus L_k^2(X,\fg_P\otimes\Lambda^{1})\oplus L_k^2(X,\fg_P\otimes\Lambda^{2,+})\oplus L_k^2(X, \fg_P)\to \mathcal{C}_{k-1}'(P),$$
	\begin{equation}
    	(\delta\tau,a,b,c)\mapsto \begin{pmatrix}
				d_A^*b+d_Ac-[B\centerdot a]-[C,a]\\
				d_A^+a+\frac{1}{4}[B\centerdot b]+\frac{1}{2}[b,C]+\frac12[B,c]+\tau b+(\delta\tau) B\\
			\end{pmatrix}.  \label{cc}
	\end{equation}
The linearization of the action of gauge group $\mathcal{G}_{k+1}(P)$ at a point $(\tau,A,B,C)\in C^r(X, \gl(\Lambda^{2,+}))\times \mathcal{C}_k(P)$
is given by
$$
d_{(\tau,A,B,C)}^0:L_{k+1}^2(X,\fg_P)\to C^r(X, \gl(\Lambda^{2,+}))\oplus L_k^2(X,\fg_P\otimes\Lambda^{1})\oplus L_k^2(X,\fg_P\otimes\Lambda^{2,+})\oplus L_k^2(X, \fg_P), 
$$
	\begin{equation}		
		\xi \mapsto(0,d_A\xi,[B,\xi],[C,\xi]).  \label{ae}
	\end{equation}
The $L^2$-adjoint of $d_{(\tau,A,B,C)}^0$ is
\begin{equation}
    d_{(\tau,A,B,C)}^{0,*}(\delta\tau,a,b,c)=d_A^*a+[b\cdot B]+[c,C]. \label{af}
\end{equation}
Since $(dF)_{(\tau,A,B,C)}+d_{(\tau,A,B,C)}^{0,*}$ is a family of first order elliptic operators depending on $\delta\tau$ linearly, $(dF)_{(\tau,A,B,C)}$ has closed image. Assume $F(\tau,A,B,C)=0$ and $B$ has rank 3. By the regularity theorem for elliptic equations with $C^r$ coefficients (cf. \cite{FL}, Proposition 3.7 and Remark 5.1), we may assume that $(A,B,C)$ is a $C^r$ solution of the perturbed Vafa-Witten equations \eqref{ab}.   To prove that $(dF)_{(\tau,A,B,C)}$  is surjective, it suffices to show that $(\mathrm{Im}(dF)_{(\tau,A,B,C)})^\perp=0$. Suppose $(\phi,\psi)\in(\mathrm{Im}(dF)_{(\tau,A,B,C)})^\perp$. Then we have
\begin{equation}
    \langle d_A^*b+d_Ac-[B\centerdot a]-[C,a],\phi\rangle_{L^2}+\langle d_A^+a+\frac{1}{4}[B\centerdot b]+\frac{1}{2}[b,C]+\frac12[B,c]+\tau b+(\delta\tau) B,\psi\rangle_{L^2}=0 \label{ag}
\end{equation} for all $(\delta\tau,a,b,c)\in C^r(X, \gl(\Lambda^{2,+})) \oplus L_k^2(X,\fg_P\otimes\Lambda^{1})\oplus L_k^2(X,\fg_P\otimes\Lambda^{2,+})\oplus L_k^2(X, \fg_P)$.

Computing the $L^2$-adjoint operator of $(dF)_{(\tau,A,B,C)}$ from equation \eqref{ag}, we deduce that $(\phi, \psi)$ satisfy the following system of equations
\begin{equation} 
	\begin{cases} 
	 d_A^* \psi -[\phi\centerdot B] -[\phi, C] =0, & \\
	 d_A^+ \phi +\frac14 [\psi\centerdot B] -\frac12 [\psi, C] +\tau^t \psi =0, & \\
	 d_A^*\phi +\frac12 [\psi \cdot B]=0, & \\
	 \li (\delta\tau)B,\psi \ri_{L^2}=0, \ \forall\, \delta\tau \in C^r(X, \gl(\Lambda^{2,+})). & \\
	 \end{cases} \label{pp}
\end{equation}

Note that the first three equations in \eqref{pp} differ from the expression for the operator $d^{1,*}_{(B,C,A)}$ on page 29 of \cite{Ma} by one term $\tau^t \psi$, where $\tau^t$ is the transpose of $\tau$ with respect to the Riemannian metric on the bundle $\Lambda^{2,+}$. 

Since $\delta\tau\in C^r(X,\gl(\Lambda^{2,+}))$ is arbitrary, from the 4th equation of \eqref{pp}, we deduce that  $\psi =0$ on $X^{(3)}(B)$, where $B$ has pointwise rank 3. 

The first equation in \eqref{pp} implies that
\begin{equation}
	[\phi \centerdot B]+[\phi,C]=0
\end{equation}
on $X^{(3)}(B)$. Lemma \ref{BC} below implies that $\phi=0$ on $X^{(3)}(B)$.

The first three equations in \eqref{pp} is a system of  first order linear differential equations of elliptic type for $(\phi,\psi)$ with  $C^r$ coefficients. By the unique continuation theorem (cf \cite{Ar,Kaz}), we deduce that $(\phi, \psi)=(0,0)$ on $X$. This implies that $\df$ is surjective.
\end{proof}

\begin{corollary}\label{c}
    Let $\mathfrak{M}_{VW}(P)=\bigcup_{\tau}\mathcal{M}_{VW,\tau}(P)$ be the parameterized Vafa-Witten moduli space, where $\tau$ ranges over $C^r(X,\gl(\Lambda^{2,+}))$, and $\mathfrak{M}_{VW}^{(3)}(P)$ be the full rank part of the parameterized Vafa-Witten moduli space where the component $B$ has rank 3. Then $\mathfrak{M}_{VW}^{(3)}(P)$ is an (infinite dimensional) smooth manifold. The projection
    \begin{equation}
        \pi:\mathfrak{M}_{VW}^{(3)}(P)\to C^r(X, \gl(\Lambda^{2,+})), \quad  [(\tau,A,B,C)] \mapsto \tau, 
    \end{equation}
    is Fredholm of index 0.
\end{corollary}
\begin{proof}
    By Lemma 1 in \cite{DG}, when $B$ has rank 3, we have $\mathrm{Stab}(A,B,C)=Z(G)$. Lemma~\ref{sj} implies that $\mathfrak{M}_{VW}^{(3)}(P)$ is a smooth manifold. For fixed $\tau$, $\df+d_{(\tau,A,B,C)}^{0,*}$ is a first order elliptic operator of index 0. Thus the projection $\pi$ is Fredholm of index 0.
\end{proof}

\emph{Proof of Theorem \ref{m}.} By Corollary \ref{c} and the Sard-Smale theorem (cf. \cite{MS}, Theorem A.5.1), for generic perturbation $\tau$ in the sense of second category, the full rank part of the perturbed Vafa-Witten moduli space $\mathcal{M}_{VW,\tau}^{(3)}(P)$ is a smooth 0-dimensional manifold. The orientation of the moduli space is established in \cite{DG} and \cite{JTU}. \qed

\begin{lemma} \label{BC}
    Suppose that $B\in\mathfrak{su}(2)\otimes\Lambda^{2,+}\R^4$ has rank 3, and $C\in\mathfrak{su}(2)$. Then the linear map
    \begin{align*}
        L_{B,C}:\mathfrak{su}(2)\otimes\Lambda^1\R^4&\to \mathfrak{su}(2)\otimes\Lambda^1\R^4\\
        \phi&\mapsto[B\centerdot\phi]+[C,\phi]
    \end{align*}
    is an isomorphism.
\end{lemma}
\begin{proof}
    According to the singular value decomposition for $\mathfrak{su}(2)\otimes\Lambda^{2,+}\mathbb{R}^4$ in  \cite[\S 4.1.1]{Ma}, there exist an oriented orthonormal basis $\{e^1,e^2,e^3,e^4\}$ for $\mathbb{R}^4$ and an orthonormal basis $\{\eta_1,\eta_2,\eta_3\}$ for $\mathfrak{su}(2)$   such that $[\eta_i,\eta_j]=2 \varepsilon_{ijk}\eta_k$, where $\varepsilon$ is the symbol anti-symmetric in the indices $i,j,k\in \{ 1,2,3 \}$, and  $B=B_{1}\eta_1\otimes(e^1\wedge e^2+e^3\wedge e^4)+B_{2}\eta_2\otimes(e^1\wedge e^3+e^4\wedge e^2)+B_{3}\eta_3\otimes(e^1\wedge e^4+e^2\wedge e^3)$, 
		where $B_{1},B_{2},B_{3}\in\R$. Let 
        \begin{equation}
            \begin{aligned}
\phi=&(\phi_{11}\eta_1+\phi_{12}\eta_2+\phi_{13}\eta_3)e^{1}+(\phi_{21}\eta_1+\phi_{22}\eta_2+\phi_{23}\eta_3)e^{2}\\
&+(\phi_{31}\eta_1+\phi_{32}\eta_2+\phi_{33}\eta_3)e^{3}+(\phi_{41}\eta_1+\phi_{42}\eta_2+\phi_{43}\eta_3)e^{4},
\end{aligned}
        \end{equation}
         where $\phi_{ij}\in\R$, $i=1,2,3,4$, $j=1,2,3$, $C=C_1\eta_1+C_2\eta_2+C_3\eta_3$, $C_{1},C_{2},C_{3}\in\R$. Consider
        \begin{align*}
            0=&[B\centerdot\phi]+[C,\phi]\\
=&\big(-[B_1\eta_1,\phi_{21}\eta_1+\phi_{22}\eta_2+\phi_{23}\eta_3]-[B_2\eta_2,\phi_{31}\eta_1+\phi_{32}\eta_2+\phi_{33}\eta_3]\\
&-[B_3\eta_3,\phi_{41}\eta_1+\phi_{42}\eta_2+\phi_{43}\eta_3]\big)e^1+\big([B_1\eta_1,\phi_{11}\eta_1+\phi_{12}\eta_2+\phi_{13}\eta_3]\\
&+[B_2\eta_2,\phi_{41}\eta_1+\phi_{42}\eta_2+\phi_{43}\eta_3]-[B_3\eta_3,\phi_{31}\eta_1+\phi_{32}\eta_2+\phi_{33}\eta_3]\big)e^2\\
&+\big(-[B_1\eta_1,\phi_{41}\eta_1+\phi_{42}\eta_2+\phi_{43}\eta_3]+[B_2\eta_2,\phi_{11}\eta_1+\phi_{12}\eta_2+\phi_{13}\eta_3]\\
&+[B_3\eta_3,\phi_{21}\eta_1+\phi_{22}\eta_2+\phi_{23}\eta_3]\big)e^3+\big([B_1\eta_1,\phi_{31}\eta_1+\phi_{32}\eta_2+\phi_{33}\eta_3]\big)\\
&-[B_2\eta_2,\phi_{21}\eta_1+\phi_{22}\eta_2+\phi_{23}\eta_3]+[B_3\eta_3,\phi_{11}\eta_1+\phi_{12}\eta_2+\phi_{13}\eta_3]\big)e^4\\
&+[C_1\eta_1,(\phi_{11}\eta_1+\phi_{12}\eta_2+\phi_{13}\eta_3)e^{1}+(\phi_{21}\eta_1+\phi_{22}\eta_2+\phi_{23}\eta_3)e^{2}\\
&+(\phi_{31}\eta_1+\phi_{32}\eta_2+\phi_{33}\eta_3)e^{3}+(\phi_{41}\eta_1+\phi_{42}\eta_2+\phi_{43}\eta_3)e^{4}]\\
&+[C_2\eta_2,(\phi_{11}\eta_1+\phi_{12}\eta_2+\phi_{13}\eta_3)e^{1}+(\phi_{21}\eta_1+\phi_{22}\eta_2+\phi_{23}\eta_3)e^{2}\\
&+(\phi_{31}\eta_1+\phi_{32}\eta_2+\phi_{33}\eta_3)e^{3}+(\phi_{41}\eta_1+\phi_{42}\eta_2+\phi_{43}\eta_3)e^{4}]\\
&+[C_3\eta_3,(\phi_{11}\eta_1+\phi_{12}\eta_2+\phi_{13}\eta_3)e^{1}+(\phi_{21}\eta_1+\phi_{22}\eta_2+\phi_{23}\eta_3)e^{2}\\
&+(\phi_{31}\eta_1+\phi_{32}\eta_2+\phi_{33}\eta_3)e^{3}+(\phi_{41}\eta_1+\phi_{42}\eta_2+\phi_{43}\eta_3)e^{4}]\\
=&2\big((-B_2\phi_{33}+B_3\phi_{42}+C_2\phi_{13}-C_3\phi_{12})\eta_1+(B_1\phi_{23}-B_3\phi_{41}-C_1\phi_{13}+C_3\phi_{11})\eta_2\\
&+(-B_1\phi_{22}+B_2\phi_{31}+C_1\phi_{12}-C_2\phi_{11})\eta_3\big)e^1\\
&+2\big((B_2\phi_{43}+B_3\phi_{32}+C_2\phi_{23}-C_3\phi_{22})\eta_1+(-B_1\phi_{13}-B_3\phi_{31}-C_1\phi_{23}+C_3\phi_{21})\eta_2\\
&+(B_1\phi_{12}-B_2\phi_{41}+C_1\phi_{22}-C_2\phi_{21})\eta_3\big)e^2\\
&+2\big((B_2\phi_{13}-B_3\phi_{22}+C_2\phi_{33}-C_3\phi_{32})\eta_1+(B_1\phi_{43}+B_3\phi_{21}-C_1\phi_{33}+C_3\phi_{31})\eta_2\\
&+(-B_1\phi_{42}-B_2\phi_{11}+C_1\phi_{32}-C_2\phi_{31})\eta_3\big)e^3\\
&+2\big((-B_2\phi_{23}-B_3\phi_{12}+C_2\phi_{43}-C_3\phi_{42})\eta_1+(-B_1\phi_{33}+B_3\phi_{11}-C_1\phi_{43}+C_3\phi_{41})\eta_2\\
&+(B_1\phi_{32}+B_2\phi_{21}+C_1\phi_{42}-C_2\phi_{41})\eta_3\big)e^4,
        \end{align*}
      and  we get a system of linear equations of order 12 with respect to $\phi_{ij}$. The determinant is
\begin{equation}
\begin{aligned}
    &\left|
\begin{array}{cccccccccccc}
 0 & -C_3 & C_2 & 0 & 0 & 0 & 0 & 0 & -B_2 & 0
   & B_3 & 0 \\
 C_3 & 0 & -C_1 & 0 & 0 & B_1 & 0 & 0 & 0 &
   -B_3 & 0 & 0 \\
 -C_2 & C_1 & 0 & 0 & -B_1 & 0 & B_2 & 0
   & 0 & 0 & 0 & 0 \\
 0 & 0 & 0 & 0 & -C_3 & C_2 & 0 & B_3 & 0 & 0 &
   0 & B_2 \\
 0 & 0 & -B_1 & C_3 & 0 & -C_1 & -B_3 & 0
   & 0 & 0 & 0 & 0 \\
 0 & B_1 & 0 & -C_2 & C_1 & 0 & 0 & 0 & 0 &
   -B_2 & 0 & 0 \\
 0 & 0 & B_2 & 0 & -B_3 & 0 & 0 & -C_3 &
   C_2 & 0 & 0 & 0 \\
 0 & 0 & 0 & B_3 & 0 & 0 & C_3 & 0 & -C_1 & 0 &
   0 & B_1 \\
 -B_2 & 0 & 0 & 0 & 0 & 0 & -C_2 & C_1 & 0 & 0
   & -B_1 & 0 \\
 0 & -B_3 & 0 & 0 & 0 & -B_2 & 0 & 0 & 0 & 0 &
   -C_3 & C_2 \\
 B_3 & 0 & 0 & 0 & 0 & 0 & 0 & 0 & -B_1 & C_3 &
   0 & -C_1 \\
 0 & 0 & 0 & B_2 & 0 & 0 & 0 & B_1 & 0 & -C_2 &
   C_1 & 0 \\
\end{array}
\right|\\
=&16\left(B_1^2B_2^2B_3^2+C_1^2B_2^2B_3^2+B_1^2C_2^2B_3^2+B_1^2B_2^2C_3^2\right)^2>0.
\end{aligned}
\end{equation}
	Since $B$ has rank 3, $B_1B_2B_3\neq0$, we have $\phi =0$. It follows that $L_{B,C}$ is injective, and is an isomorphism.  
\end{proof}

\end{document}